\documentclass[10pt,reqno]{amsart}
\usepackage{amsmath,amssymb}
\usepackage{enumitem}
\usepackage{graphics,color}
\usepackage{epsfig}
\usepackage{graphicx}
\usepackage{amsfonts,amsmath,amssymb,amscd}
\setbox0=\hbox{$+$}
\newdimen\plusheight
\plusheight=\ht0
\def\+{\;\lower\plusheight\hbox{$+$}\;}

\setbox0=\hbox{$-$}
\newdimen\minusheight
\minusheight=\ht0
\def\-{\;\lower\minusheight\hbox{$-$}\;}

\setbox0=\hbox{$\cdots$}
\newdimen\cdotsheight
\cdotsheight=\plusheight
\def\cds{\lower\cdotsheight\hbox{$\cdots$}}

\renewcommand{\(}{\left\(}
\renewcommand{\)}{\right\)}

\renewcommand{\pmod}[1]{\,(\textup{mod}\,#1)}
\theoremstyle{plain}
\newtheorem{theorem}{Theorem}[section]

\theoremstyle{definition}

\theoremstyle{remark}
\newtheorem{remark}[theorem]{Remark}
\numberwithin{equation}{section}

\begin{document}
\title{ A $q$-Analogue of a Supercongruence Related to Van Hamme's (B.2) Supercongruence  }
\author{Liton Karmakar and Arijit Jana}

\address{Department of Mathematics, National Institute of Technology Silchar, Assam - 788010, India}
\email{litonofficial8638@gmail.com}

\address{Department of Mathematics, National Institute of Technology Silchar, Assam - 788010, India}
\email{jana94arijit@gmail.com}

\begin{abstract}
	Motivated by the recent work of Li and Wang on parametric generalizations of Van Hamme's $(C.2)$ supercongruence in the $q$-setting, we establish $q$-analogues of a supercongruence related to Van Hamme's $(B.2)$ supercongruence, recently obtained by the authors. In particular, we derive parametric extensions of these $q$-supercongruences by constructing suitable pairs of hypergeometric functions through the $q$-WZ method.

\end{abstract}

\maketitle
\section{Introduction and statement of results}
For complex numbers $a$ and $q$ with $|q|<1$, the $q$-shifted factorial is defined by $$(a;q)_0:=1, \text{
for~} n\geq 1, (a;q)_n:=\prod_{k=0}^{n-1}(1-aq^k), \text{~and~}(a;q)_\infty := \prod_{k=0}^{\infty}(1-aq^k).$$
The $q$-shifted factorial for negative index is defined as
\begin{center}
    $ (a;q)_{-n}:= \dfrac{1}{\left(1- \frac{a}{q}\right) \left(1-\frac{a}{q^2}\right) \cdots \left(1-\frac{a}{q^{-n}}\right)}.$
\end{center}
The $m$-th cyclotomic polynomial $\Phi_m(q)$ is defined by
$$\Phi_m(q):=\prod_{\substack{1\leq j\leq m\\ \gcd(j,m)=1}}
\left(q-\mu^j\right),$$
where $\mu$ denotes a primitive $m$-th root of unity. It is well known that
$\Phi_m(q)$ is a polynomial in $q$ with integer coefficients. Furthermore,
$$\prod_{\substack{d\mid m\\ d>1}}\Phi_d(q)=[m].$$
The $q$-integer $[m]=[m]_q$ is defined by $
[m]=1+q+q^2+\cdots+q^{\,m-1}
=\frac{1-q^m}{1-q}.$
\par In $1997$, Van Hamme \cite{hamme} conjectured that the following Ramanujan-type series for $1/\pi$,
\begin{equation}\label{ident1}
\sum_{n=0}^{\infty} (-1)^n(4n+1)
\frac{\left(\frac{1}{2}\right)_n^3}{n!^3}
=
\frac{2}{\pi},
\end{equation}
admits the following remarkable $p$-adic analogue:
\begin{equation}\label{padic1}
(B.2)\qquad
\sum_{n=0}^{(p-1)/2}
(-1)^n(4n+1)
\frac{\left(\frac{1}{2}\right)_n^3}{n!^3}
\equiv
(-1)^{\frac{p-1}{2}}p
\pmod{p^3},
\end{equation}
where $(a)_n=a(a+1)\cdots(a+n-1)$ denotes the Pochhammer symbol. The supercongruence \eqref{padic1} was first proved by Mortenson \cite{mortenson} using a ${}_6F_5$ hypergeometric identity. Subsequently, alternative proofs were given by Zudilin \cite{zudilin} via the WZ method \cite{wilf}, and by Long \cite{long1} using hypergeometric series identities and evaluations.
Later, employing the $q$-WZ method, Guo \cite{Guo2018 JMAA} established the following $q$-analogue of \eqref{padic1}: for any positive odd integer $m$,
\begin{equation}\label{guoq}
\sum_{n=0}^{(m-1)/2}
(-1)^n [4n+1]
\frac{(q;q^2)_n^3}{(q^2;q^2)_n^3}
q^{n^2}
\equiv
(-q)^{(m-1)^2/4}[m]
\pmod{[m]\Phi_m(q)^2}.
\end{equation}


In \cite{guo}, Guo also established the following generalization of \eqref{padic1} by employing the powerful WZ method:
\begin{equation}\label{eq2}
\sum_{n=0}^{\frac{p-1}{2}}
(-1)^n(4n+1)^3
\frac{\left(\frac{1}{2}\right)_n^3}{n!^3}
\equiv
-3p(-1)^{\frac{p-1}{2}}
\pmod{p^2}.
\end{equation}
In the same paper, Guo further conjectured a family of supercongruences related to \eqref{eq2} as follows: for any odd prime $p$ and positive odd integer $h,$ there exists an integer $a_h$ such that, for any positive integer $u,$ there hold 

\begin{equation} \label{Con eqn 1}
    \displaystyle \sum_{n=0}^{\frac{p^{u}-1}{2}} (-1)^n(4n+1)^h \frac{\left(\frac{1}{2}\right)_n^3}{n!^3} \equiv a_h p^{u} (-1)^{\frac{(p-1)u}{2}}\pmod{p^{u+2}},
\end{equation}
where $a_1=1, a_3=-3, a_5=41, a_7=-1595, a_9=124689, a_{11}=-16253107.$ The cases $u=1, h=3, h=5$ and $h=7$ of \eqref{Con eqn 1} have been proved by Guo \cite{guo} using certain hypergeometric series identities and the WZ-method.  
Subsequently, Jana and Kalita \cite{jana1} obtained the following extension of \eqref{eq2}, thereby confirming the case $h=3$ of Conjecture~\eqref{Con eqn 1}: for any positive integer $u$,

\begin{equation}\label{eq3}
\sum_{n=0}^{\frac{p^u-1}{2}}
(-1)^n(4n+1)^3
\frac{\left(\frac{1}{2}\right)_n^3}{n!^3}
\equiv
-3p^u(-1)^{\frac{(p-1)u}{2}}
\pmod{p^{u+2}}.
\end{equation}
If Conjecture~\eqref{Con eqn 1} holds for both $h=1$ and $h=3$, then one observes that
\begin{equation}\label{remark 1}
\sum_{n=0}^{\frac{p^{u}-1}{2}}
(-1)^n(4n+1)(4n^2+2n+1)
\frac{\left(\frac{1}{2}\right)_n^3}{n!^3}
\equiv 0
\pmod{p^{u+2}},
\end{equation}
since
$(4n+1)^3+3(4n+1)=4(4n+1)(4n^2+2n+1).$
Motivated by this observation, Guo \cite{guo} proposed a considerably stronger supercongruence conjecture related to \eqref{remark 1}; see \cite[Conjecture~4.2]{guo}. The conjecture was later proved by Jana and Kalita \cite{jana8} in 2021.
In 2025, the authors \cite{Liton2025 RJ} established a parametric extension of \eqref{remark 1}, which may be viewed as a refinement of \cite[Conjecture~4.2]{guo} involving an additional parameter $s$. Their main result can be stated as follows.

\begin{theorem} [Theorem 1.1, \cite{Liton2025 RJ}] \label{thmjk1}
Let $\ell \geq 2$ and $p$ be an odd prime with $p \equiv -1 \pmod{\ell}.$ Let $s$ be a non-negative integer such that $p > 2\ell s - 1.$ Then
\begin{align*}
&\sum_{n=s}^{\frac{(\ell-1)p-1}{\ell}+s}
(-1)^n (2\ell n+1)
\bigl(\ell^2n^2+\ell n-\ell^2s^2+1\bigr)
\frac{\left(\frac{1}{\ell}\right)_{n+s}
      \left(\frac{1}{\ell}\right)_{n-s}
      \left(\frac{1}{\ell}\right)_n}
     {(1)_{n+s}(1)_{n-s}(1)_n}
\equiv 0 \pmod{p^3}.
\end{align*}
\end{theorem}
In the same paper, the authors derived several parametric generalizations of the Ramanujan-type supercongruence (labeled $(C.2)$) using the Zeilberger algorithm. Recently, Li and Wang \cite{LiWang C2} established $q$-analogues of these supercongruences, thereby extending the supercongruence related to $(C.2)$ to the setting of $q$-series.
\par Inspired by their work\rq s, we here establish an another new parametric extension of \eqref{remark 1} as follows: let $\ell \geq 2$ and $t$ be integers, and $u,s$ non-negative integers satisfying $p^{u} \equiv -t \pmod{\ell}$ and $p^u > 2 \ell s - t.$ Then
\begin{align} \label{Main Ext}
		&\displaystyle \sum_{n=s}^{\frac{\left(\ell - 1\right)p^{u} -t}{\ell}+s} \left(-1\right)^{n} \left(2\ell n +t\right) \left(\ell^2n^2 + t \ell n -\ell^2 s^2 + t^2 \right) \frac{\left(\frac{t}{\ell}\right)_{n+s} \left(\frac{t}{\ell}\right)_{n-s} \left(\frac{t}{\ell}\right)_{n}}{\left(1\right)_{n+s} \left(1\right)_{n-s} \left(1\right)_{n}}
		\equiv 0 \pmod{p^{u+2}}.
	\end{align}
\begin{remark}
   If we set $u=1$ and $t=1$ in \eqref{Main Ext}, then congruence \eqref{Main Ext} reduces to Theorem \ref{thmjk1}.
\end{remark}
The main objective of this paper is to prove congruence \eqref{Main Ext}. For this, we shall deduce a $q$-supercongruence in the following theorem.

\begin{theorem} \label{Result 3}
	Let $t$ and $\ell \geq 2 $ be integers, and $m$ a positive integer with $m \equiv -t \pmod{\ell}$ and \textup{gcd}$\left(m,\ell\right)=1$. Let $s$ be a non-negative integer such that $m > 2 \ell s - t.$ Then 
	\begin{align*} 
		\displaystyle  \sum_{n=s}^{\frac{\left(\ell - 1\right)m -t}{\ell}+s} & (-1)^n [2\ell n +t]
\frac{(q^t;q^\ell)_{n+s} (q^t;q^\ell)_{n-s} (q^t;q^\ell)_{n}}
{(q^\ell;q^\ell)_{n+s} (q^\ell;q^\ell)_{n-s} (q^\ell;q^\ell)_{n}} \notag \\
& \times \{q^{n \left(\frac{\ell}{2}  \left(n+1\right) -t \right)} \left[t+\ell s\right]\left[t-\ell s\right] + q^{\left(n-1\right) \left(\frac{\ell n}{2} -t \right)} \left[\ell n\right]\left[t+\ell n\right] \} \equiv  0 \pmod {[m]\Phi_m(q)^2}.
	\end{align*} 
\end{theorem}
Setting $q \to 1$ and $m=p^u$ in the above congruence, then we arrive at the congruence \eqref{Main Ext}.\\
\par \quad The paper is organized as follows. In Section~2, we derive a closed-form summation formula by constructing an appropriate $q$-WZ pair of rational functions via the $q$-WZ method of Wilf and Zeilberger \cite{wilf}. This summation formula serves as the main tool in the proof of our principal theorem.
\section{Proof of Theorem \ref{Result 3} }
To prove Theorem~\ref{Result 3}, we first establish the following Theorem.
\begin{theorem}\label{Result 1}
	Let $\ell \geq 2$,   $\epsilon \geq s\geq  0$ and $t$ be integers. Then
	\begin{align*}
		\displaystyle \sum_{n=s}^{\epsilon} &\left(-1\right)^{n} \left[2\ell n +t\right]  \frac{\left(q^t;q^\ell\right)_{n+s} \left(q^t;q^\ell\right)_{n-s} \left(q^t;q^\ell\right)_{n}}{\left(q^\ell;q^\ell\right)_{n+s} \left(q^\ell;q^\ell\right)_{n-s} \left(q^\ell;q^\ell\right)_{n}}\\
        \times & \{q^{n \left(\frac{\ell}{2}  \left(n+1\right) -t \right)} \left[t+\ell s\right]\left[t-\ell s\right] + q^{\left(n-1\right) \left(\frac{\ell n}{2} -t \right)} \left[\ell n\right]\left[t+\ell n\right] \} \\
		&= \left(-1\right)^{\epsilon} \frac{ q^{\epsilon \left(\frac{\ell}{2}  \left(\epsilon+1\right) -t \right)} \left[t\right]^3 \left(q^{\ell+t};q^\ell\right)_{\epsilon+s} \left(q^{\ell+t};q^\ell\right)_{\epsilon-s} \left(q^{\ell+t};q^\ell\right)_{\epsilon}}{\left(q^{\ell};q^\ell\right)_{\epsilon+s} \left(q^{\ell};q^\ell\right)_{\epsilon-s} \left(q^{\ell};q^\ell\right)_{\epsilon}}.
	\end{align*}
	
\end{theorem}
\begin{proof}
     For $\ell \geq 2$ and, integers $n,s,k \geq 0$ satisfying
$n \geq s$ and $n \geq k,$ we consider the following two hypergeometric function $\mathcal{P}(n,k)$ and $\mathcal{Q}(n,k)$ in $n$ and $k;$
	\begin{center}
		$\displaystyle \mathcal{P}(n,k)=  (-1)^{n+k}\left[2\ell n+t \right] \frac{{q^{\left(n-k\right) \left(\frac{\ell}{2}  \left(n-k+1\right) -t \right)}\left (q^t;q^\ell \right)}_{n+s}{\left (q^t;q^\ell \right)}_{n-s}{\left (q^t;q^\ell \right)}_{n+k}}{{\left (q^\ell;q^\ell \right)}_{n+s}{\left (q^\ell;q^\ell \right)}_{n-s}{\left (q^\ell;q^\ell \right)}_{n-k}{\left (q^t;q^\ell \right)}_{k+s}{\left (q^t;q^\ell \right)}_{k-s}}$
	\end{center}
	and
	\begin{center}
		$\displaystyle \mathcal{Q}(n,k)=  (-1)^{n+k} \frac{{q^{\left(n-k\right) \left(\frac{\ell}{2}  \left(n-k+1\right) -t \right)}\left (q^t;q^\ell \right)}_{n+s}{\left (q^t;q^\ell \right)}_{n-s}{\left (q^t;q^\ell \right)}_{n+k-1}}{{\left(1-q\right)\left (q^\ell;q^\ell \right)}_{n+s-1}{\left (q^\ell;q^\ell \right)}_{n-s-1}{\left (q^\ell;q^\ell \right)}_{n-k}{\left (q^t;q^\ell \right)}_{k+s}{\left (q^t;q^\ell \right)}_{k-s}},$	
	\end{center}
	where $1 \slash \left(q^\ell;q^\ell\right)_{y}=0$ for $y < 0.$ 
The motivation of taking the above $q$-WZ pair is based on the work of \cite{Guo2018 JMAA, Guo2019 RJ}. Then we see that
	
	\begin{align*}
	\displaystyle \frac{\mathcal{P}(n,k)}{\mathcal{Q}(n,k)}
=
\frac{(1-q^{2\ell n +t})(1-q^{t+\ell(n+k-1)})}
     {(1-q^{\ell(n+s)})(1-q^{\ell(n-s)})},
	\end{align*}
	\begin{align*}
	\displaystyle \frac{\mathcal{Q}(n+1,k)}{\mathcal{Q}(n,k)}
=
-\frac{
q^{\ell(n-k+1)-t}
(1-q^{t+\ell(n+s)})
(1-q^{t+\ell(n-s)})
(1-q^{t+\ell(n+k-1)})
}{
(1-q^{\ell(n+s)})
(1-q^{\ell(n-s)})
(1-q^{\ell(n-k+1)})
},
	\end{align*}	
	and
	\begin{align*}
	\frac{\mathcal{P}(n,k-1)}{Q(n,k)}=
    -\frac{
q^{\ell(n-k+1)-t}
(1-q^{2\ell n +t})
(1-q^{t+\ell(k+s-1)})
(1-q^{t+\ell(k-s-1)})
}
{
(1-q^{\ell(n+s)})
(1-q^{\ell(n-s)})
(1-q^{\ell(n-k+1)})
}.
	\end{align*}
From above we have the following connection:
	\begin{center}
		$\mathcal{P}(n,k-1) - \mathcal{P}(n,k)= \mathcal{Q}(n+1,k) - \mathcal{Q}(n,k).$
	\end{center}
	Pre-multiplying both sides of the above Equation by $\frac{\left(q^t;q^\ell\right)_{k+s} \left(q^t;q^\ell\right)_{k-s}}{\left(1-q\right)^2},$ we obtain
	\begin{align*}
	\displaystyle \frac{\left(q^t;q^\ell\right)_{k+s} \left(q^t;q^\ell\right)_{k-s}}{\left(1-q\right)^2} \left\{\mathcal{P}(n,k-1) - \mathcal{P}(n,k)\right\} & = \frac{\left(q^t;q^\ell\right)_{k+s} \left(q^t;q^\ell\right)_{k-s}}{\left(1-q\right)^2} \left\{\mathcal{Q}(n+1,k) - \mathcal{Q}(n,k)\right\}.
	\end{align*}
    Now, setting $k=1$ and summing both sides over $n$ from $s$ to $\epsilon$, we get
	\begin{align*}
	\displaystyle  \sum_{n=s}^{\epsilon} \frac{\left(q^t;q^\ell\right)_{1+s} \left(q^t;q^\ell\right)_{1-s}}{\left(1-q\right)^2} \left\{\mathcal{P}(n,0) - \mathcal{P}(n,1)\right\}  = \frac{\left(q^t;q^\ell\right)_{1+s} \left(q^t;q^\ell\right)_{1-s}}{\left(1-q\right)^2} \left\{\mathcal{Q}(\epsilon+1,1) - \mathcal{Q}(s,1)\right\}. 
	\end{align*}
		The left-hand side of the above expression implies  
		\begin{align*}
			&\displaystyle  \sum_{n=s}^{\epsilon} \frac{\left(q^t;q^\ell\right)_{1+s} \left(q^t;q^\ell\right)_{1-s}}{\left(1-q\right)^2} \left\{\mathcal{P}(n,0) - \mathcal{P}(n,1)\right\} \\
			& =\sum_{n=s}^{\epsilon}  \frac{{(-1)^{n}\left[2\ell n+t \right]\left (q^t;q^\ell \right)}_{n+s}{\left (q^t;q^\ell \right)}_{n-s}{\left (q^t;q^\ell \right)}_{n}}{{\left (q^\ell;q^\ell \right)}_{n+s}{\left (q^\ell;q^\ell \right)}_{n-s}{\left (q^\ell;q^\ell \right)}_{n}} \\  & \times  \left\{\frac{{q^{n \left(\frac{\ell}{2}  \left(n+1\right) -t \right)}\left (q^t;q^\ell \right)}_{1+s}{\left (q^t;q^\ell \right)}_{1-s}}{\left(1-q\right)^2{\left (q^t;q^\ell \right)}_{s}{\left (q^t;q^\ell \right)}_{-s}} \right. + \left. \frac{{q^{\left(n-1\right) \left(\frac{\ell n}{2} -t \right)}\left (q^t;q^\ell \right)}_{n+1}{\left (q^\ell;q^\ell \right)}_{n}}{\left(1-q\right)^2{\left (q^\ell;q^\ell \right)}_{n-1}{\left (q^t;q^\ell \right)}_{n}}  \right\} \\
			& = \displaystyle \sum_{n=s}^{\epsilon} \left(-1\right)^{n} \left[2\ell n +t\right] \frac{\left(q^t;q^\ell\right)_{n+s} \left(q^t;q^\ell\right)_{n-s} \left(q^t;q^\ell\right)_{n}}{\left(q^\ell;q^\ell\right)_{n+s} \left(q^\ell;q^\ell\right)_{n-s} \left(q^\ell;q^\ell\right)_{n}} \\
            & \times \{q^{n \left(\frac{\ell}{2}  \left(n+1\right) -t \right)} \left[t+\ell s\right]\left[t-\ell s\right] + q^{\left(n-1\right) \left(\frac{\ell n}{2} -t \right)} \left[\ell n\right]\left[t+\ell n\right] \}.
		\end{align*}

	 Again, from right-hand side we have
	\begin{align*}
	\frac{\left(q^t;q^\ell\right)_{1+s} \left(q^t;q^\ell\right)_{1-s}}{\left(1-q\right)^2} \left\{\mathcal{Q}(\epsilon+1,1) - \mathcal{Q}(s,1)\right\} &  = \left(-1\right)^{\epsilon} \frac{{q^{\epsilon \left(\frac{\ell}{2}  \left(\epsilon+1\right) -t \right)} \left (q^t;q^\ell \right)}_{\epsilon+s+1}{\left (q^t;q^\ell \right)}_{\epsilon-s+1} {\left (q^t;q^\ell \right)}_{\epsilon+1} }{\left(1-q\right)^4 {\left (q^\ell;q^\ell \right)}_{\epsilon+s}{\left (q^\ell;q^\ell \right)}_{\epsilon-s}  {\left (q^\ell;q^\ell \right)}_{\epsilon} }   \\
	& = \left(-1\right)^{\epsilon} \frac{ q^{\epsilon \left(\frac{\ell}{2}  \left(\epsilon+1\right) -t \right)} \left[t\right]^3 \left(q^{\ell+t};q^\ell\right)_{\epsilon+s} \left(q^{\ell+t};q^\ell\right)_{\epsilon-s} \left(q^{\ell+t};q^\ell\right)_{\epsilon}}{\left(q^{\ell};q^\ell\right)_{\epsilon+s} \left(q^{\ell};q^\ell\right)_{\epsilon-s} \left(q^{\ell};q^\ell\right)_{\epsilon}}.
	\end{align*}
	This completes the proof due to the fact that $\mathcal{Q}(s,1)=0$ as $1 \slash \left(q^\ell;q^\ell\right)_{y}=0$ for $y<0.$
\end{proof}
Now, we  establish the following weaker version of Theorem~\ref{Result 3}. 
\begin{theorem} \label{Result 2}
	Let $t$ and $\ell \geq 2 $ be integers, and $m$ a positive integer with $m \equiv -t \pmod{\ell}$ and \textup{gcd}$\left(m,\ell\right)=1$. Let $s$ be a non-negative integer such that $m > 2 \ell s - t.$ Then 
	\begin{align} 
		\displaystyle  \sum_{n=s}^{\frac{\left(\ell - 1\right)m -t}{\ell}+s} & (-1)^n [2\ell n +t]
\frac{(q^t;q^\ell)_{n+s} (q^t;q^\ell)_{n-s} (q^t;q^\ell)_{n}}
{(q^\ell;q^\ell)_{n+s} (q^\ell;q^\ell)_{n-s} (q^\ell;q^\ell)_{n}} \notag \\
& \times \{q^{n \left(\frac{\ell}{2}  \left(n+1\right) -t \right)} \left[t+\ell s\right]\left[t-\ell s\right] + q^{\left(n-1\right) \left(\frac{\ell n}{2} -t \right)} \left[\ell n\right]\left[t+\ell n\right] \} \equiv  0 \pmod {\Phi_m(q)^3}. \label{E 1}
	\end{align} 
\end{theorem}

\begin{proof}
Setting $\epsilon = \frac{\left(\ell - 1\right)m -t}{\ell}+s$ in Theorem \ref{Result 1}, we obtain 

\begin{align*}
    \displaystyle & \sum_{n=s}^{\frac{\left(\ell - 1\right)m -t}{\ell}+s}  \left(-1\right)^{n} \left[2\ell n +t\right] \frac{\left(q^t;q^\ell\right)_{n+s} \left(q^t;q^\ell\right)_{n-s} \left(q^t;q^\ell\right)_{n}}{\left(q^\ell;q^\ell\right)_{n+s} \left(q^\ell;q^\ell\right)_{n-s} \left(q^\ell;q^\ell\right)_{n}} \\
    & \times \{q^{n \left(\frac{\ell}{2}  \left(n+1\right) -t \right)} \left[t+\ell s\right]\left[t-\ell s\right] + q^{\left(n-1\right) \left(\frac{\ell n}{2} -t \right)} \left[\ell n\right]\left[t+\ell n\right] \} \\
    &= \left(-1\right)^{\frac{\left(\ell - 1\right)m -t}{\ell}+s} \frac{ q^{\left(\frac{\left(\ell - 1\right)m -t}{\ell}+s \right)\left(\frac{\ell}{2}  \left(\frac{\left(\ell - 1\right)m -t}{\ell}+s+1\right) -t \right)} \left[t\right]^3 \left(q^{\ell+t};q^\ell\right)_{\frac{\left(\ell - 1\right)m -t}{\ell}+2s} \left(q^{\ell+t};q^\ell\right)_{\frac{\left(\ell - 1\right)m -t}{\ell}} \left(q^{\ell+t};q^\ell\right)_{\frac{\left(\ell - 1\right)m -t}{\ell}+s}}{\left(q^{\ell};q^\ell\right)_{\frac{\left(\ell - 1\right)m -t}{\ell}+2s} \left(q^{\ell};q^\ell\right)_{\frac{\left(\ell - 1\right)m -t}{\ell}} \left(q^{\ell};q^\ell\right)_{\frac{\left(\ell - 1\right)m -t}{\ell}+s}} \\
    & \equiv 0 \pmod {\Phi_m(q)^3}.
\end{align*}
The last expression holds because the $q$-shifted factorial $\left(q^{\ell +t}; q^\ell \right)_n$ contains a factor which is divisible by $\Phi_m(q)$ for $n \geq \frac{\left(\ell - 1\right)m -t}{\ell}$ whereas $\left(q^\ell;q^\ell\right)_n$ is divisible by $\Phi_m(q)$ if $n \geq m,$ and when $n<m$, $\left(q^\ell;q^\ell\right)_n$ is always relatively prime with $\Phi_m(q)$ for \textup{gcd}$\left(m,\ell\right)=1.$ The above congruence is true due to $m > 2\ell s -t.$
\end{proof}

We are now in a position to prove our main theorem.
\begin{proof}[Proof of Theorem \ref{Result 3}]
   By Theorem \ref{Result 2}, the $q$-supercongruence \eqref{E 1} is true modulo $\Phi_m(q)^3.$ To prove Theorem \ref{Result 3}, we just have to show that the congruence \eqref{E 1} holds modulo $[m]$ as we know that \textup{lcm}$\left([m],\Phi_m(q)^3\right)= [m] \Phi_m(q)^2.$ Equivalently, we want to prove that

    \begin{align} 
		\displaystyle  \sum_{n=0}^{\frac{\left(\ell - 1\right)m -t}{\ell}} & (-1)^{\left(n+s\right)} [2\ell n + 2 \ell s + t]
\frac{(q^t;q^\ell)_{n+2s} (q^t;q^\ell)_{n} (q^t;q^\ell)_{n+s}}
{(q^\ell;q^\ell)_{n+2s} (q^\ell;q^\ell)_{n} (q^\ell;q^\ell)_{n+s}} \notag \\
\times & \{q^{\left(n+s\right) \left(\frac{\ell}{2}  \left(n+ s + 1\right) -t \right)} \left[t+\ell s\right]\left[t-\ell s\right] + q^{\left(n + s -1\right) \left(\frac{\ell \left(n+s\right)}{2} -t \right)} \left[\left(n+s\right)\ell \right]\left[t+\left(n+s\right)\ell \right] \}  \equiv  0 \pmod {[m]}. \label{E 2}
	\end{align}
Following the proof of Theorem \ref{Result 2}, we observe that replacing the upper limit of summation of the above congruence \eqref{E 2} by $m-1$ give the following $q$-congruence:
    \begin{align} 
		\displaystyle  \sum_{n=0}^{m-1} & (-1)^{\left(n+s\right)} [2\ell n + 2 \ell s + t]
\frac{(q^t;q^\ell)_{n+2s} (q^t;q^\ell)_{n} (q^t;q^\ell)_{n+s}}
{(q^\ell;q^\ell)_{n+2s} (q^\ell;q^\ell)_{n} (q^\ell;q^\ell)_{n+s}} \notag \\
\times & \{q^{\left(n+s\right) \left(\frac{\ell}{2}  \left(n+ s + 1\right) -t \right)} \left[t+\ell s\right]\left[t-\ell s\right] + q^{\left(n + s -1\right) \left(\frac{\ell \left(n+s\right)}{2} -t \right)} \left[\left(n+s\right)\ell \right]\left[t+\left(n+s\right)\ell \right] \}  \equiv  0 \pmod {\Phi_m(q)}. \label{E 3} 
	\end{align}

    \par Let $\mu \neq 1$ be an $m$-th root of unity, not necessarily primitive. This means that $\mu$ is a primitive root of unity of degree $j$ such that $j$ divides $m$ and $j>1.$ There exists an integer $w$ with $0 \leq w \leq j-1$ and $\ell w \equiv -t \pmod{ j}.$ Let $A_{q}(n)$ stand for the following expression:
    \begin{align*}
      A_{q}(n) = &  (-1)^{\left(n+s\right)} [2\ell n + 2 \ell s + t]
\frac{(q^t;q^\ell)_{n+2s} (q^t;q^\ell)_{n} (q^t;q^\ell)_{n+s}}
{(q^\ell;q^\ell)_{n+2s} (q^\ell;q^\ell)_{n} (q^\ell;q^\ell)_{n+s}} \notag \\
\times & \{q^{\left(n+s\right) \left(\frac{\ell}{2}  \left(n+ s + 1\right) -t \right)} \left[t+\ell s\right]\left[t-\ell s\right] + q^{\left(n + s -1\right) \left(\frac{\ell \left(n+s\right)}{2} -t \right)} \left[\left(n+s\right)\ell \right]\left[t+\left(n+s\right)\ell \right] \}.
    \end{align*}
It is easy to see that for $w<n\leq j-1,$ $\dfrac{\left(q^t;q^\ell\right)_n}{\left(q^\ell;q^\ell\right)_n}$ is congruent to $0$ modulo $\Phi_j(q)$ which yields the following:
\begin{center}
    $\displaystyle \sum_{n=0}^{j-1} A_{\mu}(n) = \sum_{n=0}^{w} A_{\mu} (n)= 0,$
\end{center}
since for $n>w,$ $\left(q^t;q^\ell\right)_n$ contains a factor which vanishes at $q=\mu,$ thereby combined with the case $m=j$ in \eqref{E 3} lead us to the desired result.
Observing the relation
\begin{center}
    $\displaystyle \dfrac{A_{\mu} \left(\gamma j +n \right)}{A_{\mu} \left(\gamma j\right)}  = \lim_{q \to \mu} \dfrac{A_{q} \left(\gamma j +n \right)}{A_{q} \left(\gamma j\right)} =\dfrac{A_{\mu} \left(n \right)}{\left[2 \ell s + t \right]},$
\end{center}
we have
\begin{align*}
    \displaystyle  \sum_{n=0}^{\frac{\left(\ell - 1\right)m -t}{\ell}} A_{\mu}(n) & = \sum_{\gamma=0}^{\frac{\left(\ell - 1\right)m - \ell w-t}{\ell j}-1} \sum_{n=0}^{j-1} A_{\mu}(\gamma j + n) + \sum_{n=0}^{w} A_{\mu} \left( \frac{\left(\ell - 1\right)m - \ell w-t}{\ell } + n \right)\\
    & = \frac{1}{\left[2\ell s + t \right]}  \sum_{\gamma=0}^{\frac{\left(\ell - 1\right)m - \ell w-t}{\ell j}-1} A_{\mu} (\gamma j) \sum_{n=0}^{j-1} A_{\mu}(n) + \frac{1}{\left[2\ell s + t \right]} A_{\mu} \left(\frac{\left(\ell - 1\right)m - \ell w-t}{\ell } \right)  \sum_{n=0}^{w} A_{\mu} \left( n \right) \\
    & = 0.
\end{align*}
This implies that $\displaystyle \sum_{n=0}^{\frac{\left(\ell - 1\right)m -t}{\ell}} A_{q}(n)$ is congruent to $0$ modulo $\Phi_j(q).$ We know that every cyclotomic polynomial $\Phi_j(q)$ is irreducible in the ring $\mathbb{Z}[q],$ we establish that the left-hand side of \eqref{E 2} is congruent to $0$ modulo
\begin{center}
    $\displaystyle  \prod_{j | m , j>1} \Phi_j(q) = \left[m \right].$
\end{center}
This completes the proof.
\end{proof} 


\end{document}